\documentclass{amsart}
\usepackage{amsmath}
\usepackage{graphicx}
\usepackage{latexsym}
\usepackage{pinlabel}
\usepackage{verbatim}
\usepackage[all]{xy}

\newtheorem{thm}{Theorem}

\newtheorem{theorem}[thm]{Theorem}

\newtheorem{proposition}[thm]{Proposition}

\theoremstyle{definition}

\newtheorem{remark}[thm]{Remark}

\newcommand{\CPb}{\overline{\mathbb{CP}}{}^{2}}
\newcommand{\CP}{{\mathbb{CP}}{}^{2}}
\newcommand{\RP}{{\mathbb{RP}}{}^{3}}

\newcommand{\Z}{\mathbb{Z}}
\newcommand{\N}{\mathbb{N}}

\def \x {\times}
\def \eu{{\text{e}}}

\newcommand{\nc}{\newcommand}
\nc{\dmo}{\DeclareMathOperator}

\newtheorem*{mthm}{Main Theorem}
\newtheorem*{mthm2}{Second Main Theorem}

\dmo{\MCG}{Mod}
\dmo{\Diff}{Diff}

\begin{document}

\title[Virtual Betti numbers of symplectic fibered $4$-manifolds]
{Virtual Betti numbers and the symplectic Kodaira dimension of fibered $4$-manifolds}

\author[R. \.{I}. Baykur]{R. \.{I}nan\c{c} Baykur}
\address{Max Planck Institute for Mathematics, Bonn, Germany \newline
\indent Department of Mathematics, Brandeis University, Waltham MA, USA}
\email{baykur@mpim-bonn.mpg.de, baykur@brandeis.edu}

\begin{abstract}
We prove that if a closed oriented $4$-manifold $X$ fibers over a \linebreak $2$- or $3$-dimensional manifold, in most cases all of its virtual Betti numbers are infinite. In turn, we show that a closed oriented $4$-manifold $X$ which is not a tower of torus bundles and fibering over a $2$- or $3$-dimensional manifold does not admit a torsion symplectic canonical class, nor is of Kodaira dimension zero. 
\end{abstract}

\maketitle

\setcounter{secnumdepth}{2}
\setcounter{section}{0}

% ==========================================================================================================
\section{Introduction} 
% ==========================================================================================================

Let $X$ be a $4$-manifold which is closed, smooth, and oriented. For any $i \neq 0, 4$, the \emph{virtual Betti number} $b_i$ of $X$, denoted by $vb_i(X)$, is defined as the supremum taken over the set of $i$-th Betti numbers of finite covers of $X$, so it takes values in $\N \cup \{+\infty \}$. We define the virtual $b_2^+$ and $b_2^-$ of $X$ similarly. On the other hand, $X$ is said to be \emph{fibered} if it admits a surjective submersion onto a $k$-dimensional manifold, which therefore defines a fiber bundle over the target manifold with fibers of dimension $4-k$ by the celebrated theorem of Ehresmann, where we assume that $k$ is non-zero. We will often encode a fibration in the form $A \hookrightarrow X \stackrel{f}{\rightarrow} B$, specifying the base $B$, the fiber $A$, and the fibration map $f$ on $X$.

Our first theorem shows that virtual Betti numbers of fibered $4$-manifolds demonstrate a similar behavior to that of lower dimensional manifolds in many cases:

\begin{mthm}
Let $X$ be a closed smooth oriented $4$-manifold. If $X$ is 
\begin{itemize}
\item an $S^1$ bundle over a $3$-manifold $N$ which is a nontrivial connected sum of non-spherical $3$-manifolds, or is an irreducible $3$-manifold not covered by a torus bundle, or 
\item a genus $g$ surface bundle over a genus $h$ surface with $gh \neq 0$ or $1$,
\end{itemize}
then $vb_1(X), vb_2(X), vb_2^+(X), vb_2^-(X), vb_3(X)$ are all $+\infty$. 
\end{mthm}

\noindent \footnote{Upon sharing a draft of this article, Stefan Friedl kindly informed us that Stefano Vidussi and himself have a preprint in preparation which contains an independent proof of $vb_1= + \infty$ for surface bundles over surfaces with the obvious exceptions.}We shall note the case of $S^1$ bundles over $3$-manifolds follows rather easily using the virtual properties of $3$-manifold groups (see Proposition~\ref{virtual3} below). One can in fact extend the first case so as to cover all connected sums of $3$-manifolds but $\RP \# \RP$, leading the desired result, though for the main purposes of this note (namely proving that the virtual Betti numbers of surface bundles are infinite and the Second Main Theorem that follows) we will be content with the statement as it is. All the other hypotheses on the $4$-manifold $X$ in the statement of the theorem are necessary: When $X$ is an $S^1$ bundle over a $3$-manifold covered by a torus bundle, a $T^2$ bundle over $T^2$ (when $gh=1$), or a ruled surface over a surface (when $gh=0$), a finite index subgroup of $\pi_1(X)$ is solvable with derived series of length at most four, and so is any finite index subgroup $H$ of it, which implies that the abelianization of $H$ has rank $\leq 4$. If we remove the assumption on fiberedness, we can for instance take any simply-connected symplectic $4$-manifold $X$ which will therefore have $vb_1(X)=0$. Finally, if we remove either one of the extra assumption for the fibrations, $X=S^1 \x L(p,1)$, which fibers over both $S^2$ and $L(p,1)$, and has $\pi_1(X)= \Z \oplus \Z_p$, is an example with $vb_1(X) =1 $. 

Although a lot is known on the topology of symplectic $4$-manifolds with torsion canonical class \cite{Bauer, Li, Li1, MS}, there are very few known examples: with the exception of the $\text{K3}$ surface, all known examples fiber over $T^2$, or over $T^2$ bundles over $S^1$, and thus all are towers of torus bundles (of any dimension between $1$ and $4$). One might therefore think that fibered $4$-manifolds constitute a good pool of candidates to fish for more examples. We will show however, this is not the case:

\begin{mthm2}
A closed symplectic $4$-manifold $X$ with torsion canonical class does not fiber over a $2$- or $3$-dimensional manifold, unless it is a tower of torus bundles. 
\end{mthm2}

\noindent In the case of $X$ fibering over a $3$-manifold, the theorem is due to Friedl\,and\,Vidussi \cite{FV}. We will reproduce their result taking a slightly different path, as discussed below. 

Two different fields of mathematics will get engaged in our proof of this theorem, through the following two beautiful theorems obtained via Gauge theory and geometric group theory, respectively:

\begin{theorem}[Bauer \cite{Bauer}, Tian-Jun Li \cite{Li1}] \label{thm1}
If the canonical class $K_{\omega}$ of the symplectic $4$-manifold $(X, \omega)$ is torsion, then $b_1(X) \leq 4$. 
\end{theorem}

\begin{theorem}[Agol \cite{Agol}, Kojima \cite{Ko}, Luecke \cite{Lu}] \label{thm2}
If $N$ is a closed orientable irreducible $3$-manifold which is not a graph manifold, then it admits finite covers with arbitrarily large first Betti numbers. 
\end{theorem}

The existence of the finite coverings with arbitrarily large Betti numbers promised in the Main Theorem will be established via ``dimensional reductions'' prescribed by the fibering, where the Theorem~\ref{thm2} will play a key role. Almost in all cases, the arguments can be presented in terms of group theoretic properties of $\pi_1(X)$ of a fibered $4$-manifold $X$, which is an extension of the fundamental group of the base by that of the fiber. Nevertheless, we will describe the coverings explicitly, so as to provide more geometric insight to the situation. 

The Main Theorem will then imply the Second Main Theorem: Assume that $X$ is an $S^1$ bundle over a $3$-manifold $N$ with euler class $e$. If $e$ is torsion, $H_1(N)$ contains a non-trivial torsion subgroup, which we can realize as the image of a normal subgroup $H$ of $\pi_1(N)$ under the abelianization map. Let $\tilde{N}$ be the covering of $N$ associated to $H$. It is then easy to see that the pull-back bundle $S^1 \hookrightarrow \tilde{X} \stackrel{\tilde{f}}{\rightarrow} \tilde{N}$ has trivial Euler class, so $\tilde{X} = S^1 \x \tilde{N}$. Since $\tilde{X}$ is a finite cover of a symplectic $4$-manifold $X$, it is symplectic itself. Thus, we can assume that $X$ is an $S^1$ bundle over $N$ with non-torsion euler class $e$ to begin with. However, as shown in \cite{McCarthy, Bowden}, the existence of a symplectic form on $X$ then implies that $N$ is either $S^1 \x S^2$ or irreducible. The former would force $X$ to have a non-torsion canonical class, so we can take $N$ to be irreducible. Here $b_1(N)= b_2^+(X) + 1 \geq 2$ by the symplecticity of $X$, so $N$ can be covered by a torus bundle only if it is a tower of torus bundles \cite{Scott}, which is not possible by our assumption on $X$. On the other hand, if $X$ is a genus $g$ surface bundle over a genus $h$ surface, then $gh \neq 1$ by our assumption on the topology of $X$, and $gh \neq 0$, since otherwise we would have a ruled symplectic surface which always has non-torsion symplectic class. Hence, no matter if $X$ fibers over a $2$- or $3$-dimensional manifold, the conditions of the Main Theorem are satisfied, so we have $vb_1(X) = +\infty$. For any finite cover $\tilde{X}$ of a symplectic $4$-manifold $X$ with torsion canonical class, the pull-back symplectic form on $\tilde{X}$ will also have torsion canonical class. However by Theorem~\ref{thm1}, this is impossible, once $b_1(\tilde{X}) > 4$, completing the proof of the Second Main Theorem. 

The organization of our paper is as follows: We first prove the Main Theorem in the two cases; when $X$ is an $S^1$-bundle over a $3$-manifold or a surface bundle over a surface. We will then make a short digression into the minimality of symplectic fibered $4$-manifolds. Lastly, we will show that the statement of our Second Main Theorem holds verbatim for symplectic $4$-manifolds with Kodaira dimension zero, and discuss how to shorten our proof and improve our results in this particular case.

\vspace{0.1in}
\noindent \textit{Acknowledgements.} We would like to thank Neil Hoffman and Koji Fujiwara for helpful discussions.
We also thank Stefan Friedl and Josef Dorfmeister for very helpful comments on an earlier version of this paper. The author was partially supported by the NSF grant DMS-0906912.

\vspace{0.2in}
% ==========================================================================================================
\section{Proofs} 
% ==========================================================================================================

All manifolds that appear below are assumed to be closed, smooth, and oriented, whereas the maps are smooth. 

\vspace{0.1in}
% ==========================================================================================================
\subsection{The proof of the Main Theorem.} \
% ==========================================================================================================

Our proof will be given in two respective cases marked by the dimensions of the target of the fibration.

\vspace{0.1in}
% ==========================================================================================================
\noindent{\textbf{Fibering over a $3$-manifold.}} 
% ==========================================================================================================

A complete treatment of this case when the base $3$-manifold is irreducible, directed towards proving the corresponding statement in the Second Main Theorem, is already present in \cite{FV}. We will give a short proof using Theorem~\ref{thm2}, which immediately implies the authors' main result thanks to the strong constraints of Theorem~\ref{thm1}. This result will also be needed to handle the case of surface bundles over surfaces below.

\begin{proposition} \label{virtual3}
Let $X$ be an $S^1$ bundle over a $3$-manifold $N$. If $vb_1(N)= +\infty$, then $vb_1(X)= +\infty$.
\end{proposition}

\begin{proof}
Assume that $vb_1(N) = + \infty$. Corresponding to each covering $\tilde{N}$ of $N$ we have a covering $\tilde{X}$ of $X$, which is an $S^1$ bundle over $\tilde{N}$. Hence $b_1(\tilde{X}) \geq b_1(\tilde{N})$, and $vb_1(N) = +\infty$ implies $vb_1(X) = +\infty$.
\end{proof}

Let us now show that Proposition~\ref{virtual3} holds under the assumptions we have made in the Main Theorem. First assume that $N$ is an irreducible $3$-manifold. We claim that $N$, being an irreducible $3$-manifold not covered by a torus bundle implies that it always admits a finite cover $\tilde{N}$ with arbitrarily large $b_1$: First suppose that $N$ is a graph manifold. Then, with the above assumptions in hand, the work of Kojima in \cite{Ko} implies that $N$ can be covered by a $3$-manifold with arbitrarily large $b_1$. On the other hand, if we suppose that $N$ is not a graph manifold, we can then employ Theorem~\ref{thm2} to obtain the desired covers of $N$. 

Now if $N$ is a connected sum of non-spherical $3$-manifolds $N_1$ and $N_2$ (and possibly some others), there is a finite cover $N'$ of $N$ with summands $N'_i$ covering $N_i$ and having $b_1(N'_i) \geq 1$ for each $i=1,2$. For any $m \geq 1$, we can take a degree $m$ cover of say $N'_1$ to produce a degree $m$ covering $\tilde{N}$ of $N'$ with $m$ $N_2$ summands so that $b_1(\tilde{N}) \geq m$. Since $\tilde{N}$ is also a finite covering of $N$, we are done.

\vspace{0.1in}
% ==========================================================================================================
\noindent{\textbf{Fibering over a surface.}}
% ==========================================================================================================

Now let $F \hookrightarrow X \stackrel{f}{\rightarrow} B$ be a surface bundle over a surface with $F \cong \Sigma_g$ and $B \cong \Sigma_h$, where at least one of $g$ or $h$ is greater than one, by our assumptions on the topology of $X$.\footnote{The assumption on the symplecticity of $X$ is almost intrinsic here; unless it is a $T^2$ bundle over a $\Sigma_h$, $X$ can always be equipped with a symplectic form via the Thurston construction. We do not however assume that the given form is necessarily of Thurston type, i.e. compatible with the fibration.}

For what follows, let us first recall that $H_1(X)$ surjects onto $H_1(B)$, so $b_1(X) \geq 2h$. If $h \geq 2$, for any $n \geq 1$, we can take an $n$-fold covering $\tilde{B}$ of the base $B$ and construct a pull-back bundle $\tilde{X}$. Then $b_1(\tilde{X}) \geq b_1(\tilde{B}) = 1 + n(h -1)$, which is strictly increasing in $n$. 

So we are down to one case: $g \geq 2$ and $h=1$. Note that $X$ admits a Thurston symplectic form, and so does any covering of it. Here $X$ is uniquely determined by the monodromy homomorphism $\mu: \pi_1(T^2) \to \MCG(F)$, the mapping class group of $F \cong \Sigma_g$. In turn, $\mu$ is determined by a commuting pair of elements $A, B$ in $\MCG(F)$, the images of a pair of generators of $a,b$ for $\pi_1(T^2) \cong \Z^2$ under $\mu$. What matters for our analysis (as it will become evident shortly) is the type of these two elements $A$ and $B$; whether they are periodic, reducible, or pseudo-Anosov. For a description of these mapping class types and the basic facts we invoke below, the reader can turn to \cite{Ivanov}. 

If either one of $A$ or $B$ is periodic, say $A$, we can pass to a pull-back cover which is now an $F$ bundle over $T^2$ with monodromy $\tilde{\mu}$ prescribed by $A=1$ and $B$. Choosing identity as the representative for the first mapping class element, we can see that this cover is diffeomorphic to $S^1$ times an $F$ bundle over $S^1$, and so is an $S^1$ bundle over a $3$-manifold. Hence, without loss of generality, we can assume that neither $A$ nor $B$ is periodic (and not identity in particular) to begin with.

Since a reducible element does not commute with a pseudo-Anosov element, we are left with two possibilities: Either $A$ and $B$ are both pseudo-Anosovs or both reducible. 

In the former case, it is known that the two elements then have a common root in $\MCG(F)$, so $A= \eta ^m$, $B = \eta ^n$ for some non-zero $m, n$. If we take the finite index subgroup of $\pi_1(T^2)$ generated by the elements $ma - nb$ and $b$, the pull-back bundle corresponding to it has monodromy (factoring through $\mu$) prescribed by the pair of elements $A'=\eta^{mn-nm}=1$ and $B'=\eta^m$. So we once again have $S^1$ times an $F$ bundle over $S^1$. 

In the latter case, we can pass to a pull-back cover if necessary so that $A$ and $B$ are both pure, i.e. they fix a disjoint, non-empty collection of curves on the nose. Two pure reducible elements in $\MCG(F)$ commute only if there is at least one curve $\alpha$ on $F$ they both fix. 

Assume for now that $\alpha$ is non-separating. There is a finite covering $p\colon \tilde{F} \to F$ which lifts $\alpha$ to $n$ disjoint loops $\alpha_1, \ldots, \alpha_n$ representing distinct primitive homology classes which can be completed to a basis for $H_1(\tilde{F})$, where $n$ can be taken arbitrarily large. As shown by Morita \cite[Lemma~4.1]{Morita}, after passing to a pull-back cover (which in our case will be again over $T^2$), one can fiberwise cover the latter by a surface bundle $\tilde{F} \hookrightarrow \tilde{X} \stackrel{\tilde{f}}{\rightarrow} \tilde{B}$, where the restriction to fibers are the prescribed covering $p\colon \tilde{F} \to F$. (This can be regarded as a ``characteristic cover'' construction.) Since the monodromy of the pull-back bundle downstairs still fix $\alpha$, the one upstairs fix the disjoint union of $\alpha_1, \ldots, \alpha_n$. If needed, we can pass to a pull-back bundle $(\tilde{\tilde{X}}, \tilde{\tilde{f}})$ of $(\tilde{X}, \tilde{f})$ which fix each $\alpha_i$, and therefore fix an $n$-dimensional subspace of $H_1(\tilde{F})$ spanned by the primitive homology classes $[\alpha_i]$, for $i=1, \ldots, n$. Finally, if $\alpha$ was separating, we could first take a covering $\tilde{F} \to F$ so that the lift of $\alpha$ would be non-separating, and apply the above array of arguments. 

Now $H^0(B, H_1(F))= H^0(\pi_1(B), H_1(F)) = (H_1(F))^{\pi_1(B)}$, which by Poincar\'e duality is equal to $H_2(B, H_1(F))$. Here $(H_1(F))^{\pi_1(B)}$ denotes the stabilizer subgroup of $H_1(F)$ under the monodromy action of $\pi_1(B)$. On the other hand, from the Leray-Serre spectral sequence of a fibration, we get an exact sequence
\[ \ldots \to H_3(X) \to H_2(B, H_1(F)) \to H_0(B, H_2(F)) \to \ldots \]
The kernel of the map $H_2(B, H_1(F)) \to H_0(B, H_2(F))$ is contained in the $\pi_1(B)$ invariant subspace of $H_1(F)$, which we have shown above to be of dimension at least $n$. However, $H_0(B, H_2(F)) \cong \Z$, so we see that the rank of $H^1(X) \cong H_3(X)$ is at least $n-1$, which implies that $b_1(X) \geq n-1$. Taking $n$ arbitrarily large gives the desired covers of $X$. 

Hence we see that in all these cases, $vb_1(X) = +\infty$.

\vspace{0.1in}
% ==========================================================================================================
\noindent{\textbf{Finishing the proof.}} \

We have now proved our claim for the $vb_1(X)$. By the Poincar\'e duality \linebreak $vb_3(X) = + \infty$ as well. For the remaining Betti numbers, recall that the Euler characteristic and the signature are multiplicative under coverings. So, if we have $\eu(X)=0=\sigma(X)$, then $vb_1(X)= +\infty$ implies $vb_2(X), vb_2^+(X), vb_2^-(X)$ are all $+\infty$. This condition holds in the majority of the cases: all $S^1$ bundles over $3$-manifolds, and surface bundles with base or fiber $T^2$ (as well as all $3$-manifold bundles over $S^1$) have vanishing Euler characteristic and signature. For surface bundles with fiber and base genera at least two, if $X_m$ is a covering of $X$ obtained via an $m$-fold cover of the base, then from the Euler characteristic and the signature formulae for the covers we deduce that 
\[ b_2^+(X_m)= m (b_2^+(X)-1)+1 \, , \ \ \text{and} \ \ b_2^-(X_m)= m(b_2^-(X)-1)+1 \, .\]
So both $b_2^+(X_m)$ and $b_2^-(X_m)$ are strictly increasing in $m$, provided $b_2^+(X), b_2^-(X)$ are greater than one. Since $X$ admits a symplectic form in this case, and is minimal (see the next subsection), we have $c_1^2(X, \omega) \geq 0$, yielding
\[ 0 \leq 2 \eu(X)+ 3 \sigma(X) = 4 - 4b_1(X)+ 5b_2^+(X)-b_2^-(X) , \] 
which implies 
\[ b_2^+(X) \geq \frac{4}{5}(b_1(X) - 1) \geq \frac{4}{5}(2h-1) . \]
So $b_2^+(X) \geq 3$. The same holds when we reverse the orientation, so $b_2^-(X) \geq 3$, too. Hence $vb_2^+(X),vb_2^-(X)$, and therefore $vb_2(X)$ are all $+\infty$. 

It is easy to see that in the case of ruled surfaces, i.e. when we have a surface bundle over a surface with base or fiber $S^2$, the middle dimensional Betti numbers of the coverings do not change.

\begin{remark} \label{MappingTorusCase}
Similar results can be seen to hold for a (symplectic) $4$-manifold $X$ fibering over $S^1$ in many cases. For instance, if the associated monodromy $\mu$ of $M \hookrightarrow X \stackrel{f}{\rightarrow} S^1$ is finite, we can pass to a finite cover $\tilde{X}$ of $X$ which is a product $\tilde{M} \x S^1$, where $\tilde{M}$ is a finite cover of $M$, and apply the above results. Such a condition on the monodromy is satisfied when $M$ is an irreducible $3$-manifold with only hyperbolic pieces in its JSJ decomposition \cite{BF}. However, due to the rich structure of mapping class groups of $3$-manifolds, settling these questions for mapping tori of $3$-manifolds appears to be a more subtle task when the fibers are reducible.
\end{remark}

%\vspace{0.1in}
% ==========================================================================================================
\subsection{Minimality of symplectic fibered $4$-manifolds} \
% ==========================================================================================================

Let $X$ be a any symplectic $4$-manifold. If it is an $S^1$-bundle over a $3$-manifold $N$ with non-torsion Euler class $e$, $X$ is either aspherical or equal to $S^2 \x T^2$, following a slight modification of the argument in \cite{McCarthy} given in \cite{Bowden}. If $e$ is torsion, as discussed in the introduction, we can pass to a covering $\tilde{X}$ of $X$ which is an $S^1$ bundle over $\tilde{N}$ with trivial Euler class, where $\tilde{N}$ is a finite cover of $N$. So the above argument applies --- since $\pi_2(X)=\pi_2(\tilde{X})$. 

If $X$ is a genus $g$ surface bundle over a genus $h$ surface, the same holds, provided $g, h \geq 1$, which can be seen from the homotopy long exact sequence of the fibration. When $gh=0$, $X$ is a ruled surface, so non-minimal only if it is $S^2 \tilde{\x} S^2 \cong \CP \# \CPb$. 

Now if $X$ fibers over $S^1$, it is not aspherical unless the fiber $M$ is. Nevertheless, we can take a closer look at $H_2(X)$ via the exact sequence
\[ 1 \to \text{Coker}(f_* - 1)|_{H_2(M)} \to H_2(X) \to \text{Ker}(f_* - 1)|_{H_1(M)} \to 1 \]
which is derived from the Mayer-Vietoris sequence obtained using a splitting of the base circle into two intervals. We see that all the non-trivial classes in $H_2(X)$ are generated by the inclusions of the classes in $H_2(X)$, which are either in the cokernel of $(f_* - 1)|_{H_2(M)}$ or are dual to circles that are in the kernel of $(f_* - 1)|_{H_1(M)}$. So every class in $H_2(X)$ has even self-intersection. We conclude that

\begin{proposition}
If a fibered $4$-manifold $X$ admits a symplectic form, then it is minimal, unless $X=S^2 \tilde{\x} S^2$.
\end{proposition}

%\vspace{0.1in}
% ==========================================================================================================
\subsection{Symplectic $4$-manifolds of Kodaira dimension zero.} \
% ==========================================================================================================

Recall that a symplectic $4$-manifold $(X, \omega)$ is said to have \emph{Kodaira dimension zero}, if $[\omega] \cdot K = K \cdot K = 0$. It turns out that for minimal symplectic $4$-manifolds, these are precisely the ones with torsion symplectic canonical class \cite{Li}. So our Second Main Theorem could be rephrased for symplectic $4$-manifolds of Kodaira dimension zero, for we have seen above that symplectic fibered $4$-manifolds are minimal, except when $X=S^2 \tilde{\x} S^2$ which, along with any ruled surface, does not admit a torsion symplectic canonical class. 

If aimed to obtain this result on Kodaira dimension zero symplectic $4$-manifolds alone, only one step in our proof above could be bypassed, essentially by invoking a more advanced result. While dealing with surface bundles, we had to make a case by case analysis for commuting pairs of mapping class group elements, so as to handle genus $g \geq 2$ bundles over $T^2$. As shown by Tian-Jun Li, the Kodaira dimension of a minimal symplectic $4$-manifold is an oriented diffeomorphism invariant \cite[Theorem 2.4]{Li}. So we can run the argument for a Thurston symplectic form instead to arrive at the desired conclusion, since the canonical class of this form is non-torsion, as easily seen from the adjunction.

Also note that the proof for surface bundles can be easily extended over Lefschetz fibrations over positive genera. When $F \hookrightarrow X \stackrel{f}{\rightarrow} B$ is a genus $g$ Lefschetz fibration over a genus $h \geq 1$ surface, $H_1(X)$ still surjects onto $H_1(B)$. For $h \geq 2$, one can construct pull-back bundles as above to argue that $vb_1(X) = +\infty$. If $h=1$, since $b_1(X) > 0$, once again $\sigma(X)=0$, as observed both in \cite{Bauer} and \cite{Li}. On the other hand, the Euler characteristic calculation gives
\[ b_2^+(X) = 2(g-1)(h-1) + \frac{n}{2} + b_1(X) - 1 , \]
where $n$ is the number of critical points. So $n=0$ and we indeed have a surface bundle over a surface as before.

\begin{remark}
A similar study for surface bundles and Lefschetz fibrations is carried out in \cite{DZ}, where the authors treat the trickiest case of surface bundles over $T^2$ by invoking the subadditivity of Kodaira dimensions result from \cite{Li5}. Although apparent when the symplectic form on $X$ is compatible with the fibration, it is not immediately clear why the subadditivity works for arbitrary symplectic forms. However, the desired result follows as above by invoking Li's theorem in \cite{Li} and running the argument for a Thurston (or Thurston-Gompf in the case of a Lefschetz fibration) symplectic form instead. 
\end{remark}

\enlargethispage{1.5cm}
\begin{remark} \label{useofsymp}
Now that we have completed the proof of the Second Main Theorem, let us highlight at which points in the proof we have used the assumption that $X$ was symplectic. Whenever we reduced our analysis to the case of an $S^1$ bundle over a $3$-manifold $N$, the symplecticity assumption was invoked to show that $N$ was irreducible. A close look at McCarthy's proof in \cite{McCarthy} of this fact would reveal that this condition can be replaced by the condition that all finite covers of $X$ have non-trivial Seiberg-Witten invariants, which is certainly true when $X$ is symplectic. For surface bundles, the symplecticity assumption was rather intrinsic, since except for $T^2$ bundles over $\Sigma_h$ with $h \geq 2$, the total space $X$ of any such bundle admits a symplectic structure. In summary, it is plausible that one can replace the assumption on symplecticity by the non-vanishing of Seiberg-Witten invariants of all covers of $X$, which  \textit{a priori} is a property attained by a broader class of $4$-manifolds.
\end{remark}

%\vspace{1cm}

\end{document}